\documentclass[11pt,twoside, leqno]{article}

\usepackage{amssymb}
\usepackage{amsmath}
\usepackage{amsthm}

\allowdisplaybreaks

\def\rr{{\mathbb R}}
\def\rn{{{\rr}^n}}

\def\zz{{\mathbb Z}}
\def\nn{{\mathbb N}}

\def\cc{{\mathbb C}}
\def\cs{{\mathcal S}}

\def\cq{{\mathcal Q}}
\def\cp{{\mathcal P}}

\def\fz{\infty}
\def\az{\alpha}

\def\vz{\varphi}

\def\hs{\hspace{0.3cm}}

\def\ls{\lesssim}

\def\gfz{\genfrac{}{}{0pt}{}}

\def\rn{{{\mathbb R}^n}}
\def\rr{{\mathbb R}}
\def\cc{{\mathbb C}}
\def\zz{{\mathbb Z}}
\def\nn{{\mathbb N}}

\def\hs{\hspace{0.3cm}}

\def\fz{\infty}
\def\az{\alpha}
\def\supp{{\mathop\mathrm{\,supp\,}}}

\def\vz{\varphi}

\def\ls{\lesssim}

\def\laz{\langle}
\def\raz{\rangle}

\def\bmo{{{\mathrm {BMO}\,(\rn)}}}

\def\r{\right}
\def\lf{\left}

\newcommand{\dbt}{\dot{B}_{p,q}^{s,\tau}(\rn)}
\newcommand{\dft}{\dot{F}_{p,q}^{s,\tau}(\rn)}

\newcommand{\dbh}{B\dot{H}_{p,q}^{s,\tau}(\rn)}
\newcommand{\dfh}{F\dot{H}_{p,q}^{s,\tau}(\rn)}

\newcommand{\dsbt}{\dot{b}_{p,q}^{s,\tau}(\rn)}
\newcommand{\dsft}{\dot{f}_{p,q}^{s,\tau}(\rn)}

\newtheorem{thm}{Theorem}

\newtheorem{prop}{Proposition}
\newtheorem{rem}{Remark}
\newtheorem{cor}{Corollary}
\newtheorem{defn}{Definition}

\def\hs{\hspace{0.3cm}}

\begin{document}

\title{{\vspace{-5cm}\small\hfill\bf Appl. Anal. (to appear)}\\
\vspace{4cm}\bf\Large Relations among Besov-Type Spaces,
Triebel-Lizorkin-Type Spaces and Generalized Carleson
Measure Spaces\footnotetext{\hspace{-0.35cm} 2010 {\it
Mathematics Subject Classification}. Primary 42B35; Secondary 46E35.
\endgraf
{\it Key words and phrases}. Besov-type space, Triebel-Lizorkin-type space,
generalized Carleson measure space. \endgraf
The first author is supported by the National
Natural Science Foundation (Grant No. 11171027) of China
and Program for Changjiang Scholars and Innovative
Research Team in University of China, and the second (corresponding) author is supported by the National
Natural Science Foundation (Grant No. 11101038) of China.\endgraf
$^\ast$ Corresponding author.}}
\date{}
\author{Dachun Yang and Wen Yuan\,$^\ast$}
\maketitle

\begin{center}
\begin{minipage}{13.5cm}{\small
{\noindent{\bf Abstract}\quad
In this paper, the authors construct some counterexamples to show that
the generalized Carleson measure space and the Triebel-Lizorkin-type space
are not equivalent for certain parameters, which was claimed
to be true in [Taiwanese J. Math. 15 (2011), 919-926]. Moreover,
the authors show that for some special parameters, the generalized Carleson measure space,
the Triebel-Lizorkin-type space and the Besov-type space coincide with certain Triebel-Lizorkin space,
which answers a question posed
in Remark 6.11(i) of [Lecture Notes in Mathematics 2005,
Springer-Verlag, Berlin, 2010, xi+281 pp.]. In conclusion,
the Triebel-Lizorkin-type space and the Besov-type space
become the classical Besov spaces,
when the fourth parameter is sufficiently large.}}
\end{minipage}
\end{center}

\arraycolsep=1pt

\vspace{0.4cm}

Function spaces have
been widely used in various areas of analysis such as harmonic analysis
and partial differential equations.
In recent years, there has been increasing interest
in a new family of function spaces, called $Q_\az$ spaces
with $\az\in\rr$; see, for example, \cite{dx,ejpx,x,x06}
and their references for a history of these spaces.

On the other hand, the most known general scales of function spaces are the
scales of Besov spaces and Triebel-Lizorkin spaces.
It is well known that Triebel-Lizorkin
spaces ${\dot F}^s_{p,\,q}$ and ${F}^s_{p,\,q}$, and Besov spaces
${\dot B}^s_{p,\,q}$ and ${B}^s_{p,\,q}$ on $\rn$
respectively domains in $\rn$ for the full ranges of parameters
$s\in\rr$ and $p,\,q\in(0,\,\fz]$ were introduced between 1959 and
1975; see, for example, \cite{t83}. Moreover, it is known that
Triebel-Lizorkin spaces cover many well-known
classical concrete function spaces such as H\"older-Zygmund spaces,
Sobolev spaces, fractional Sobolev spaces (also often referred to as
Bessel-potential spaces), Hardy spaces and $\bmo$, which have their
own history. A comprehensive treatment of these function spaces and
their history can be found in Triebel's monographes \cite{t92, t06}.

Recently, Dafni and Xiao \cite{dx} introduced the Hardy-Hausdorff space
$\mathrm{HH}^1_{-\az}(\rn)$ with $\az\in(0,\min\{1,n/2\})$ and proved that
these spaces are predual spaces of $Q_\az(\rn)$. It was also asked in
\cite{dx} whether there exist some relations among $Q_\az(\rn)$, $\mathrm{HH}^1_{-\az}(\rn)$
and some classical function spaces such as Besov and Triebel-Lizorkin spaces.
To answer this question,
motivated by the Carleson measure characterization of
$Q_\az(\rn)$ spaces in \cite{dx},
we in \cite{yy1} introduced the Triebel-Lizorkin-type spaces $\dft$ with $s\in\rr$, $\tau\in[0,\fz)$, $p\in(1,\fz)$ and $q\in(1,\fz]$ and
their preduals, the Triebel-Lizorkin-Hausdorff spaces $\dfh$ with
$s\in\rr$, $p\in(1,\fz)$, $q\in[1,\fz)$ and
$\tau\in[0,\frac1{(\max\{p,\,q\})'}]$,
and proved therein that these spaces contain classical Triebel-Lizorkin spaces,
$Q$ spaces $Q_\az(\rn)$ and Hardy-Hausdorff spaces $\mathrm{HH}^1_{-\az}(\rn)$
as special cases.

We in \cite{yy2} further extended the spaces $\dft$
to all $p\in(0,\fz)$ and $q\in (0,\fz]$. Furthermore, the
Besov-type spaces $\dbt$ with $s\in\rr$, $\tau\in[0,\fz)$
and $p,\,q\in(0,\fz]$ and their preduals, the Besov-Hausdorff
spaces $\dbh$ with $s\in\rr$, $p,\,q\in[1,\fz)$, $\max\{p,\,q\}>1$
and $\tau\in[0,\frac1{(\max\{p,\,q\})'}]$,
were also introduced in \cite{yy2}. It is easy to see
that $\dbt$ and $\dbh$ cover the classical Besov spaces as special
cases. Some properties of the spaces $\dft$ and $\dbt$,
including the $\vz$-transform characterizations, Sobolev-type
embedding properties and smooth atomic and molecular decompositions of
these spaces, were also established in \cite{yy2}.

Recently, Lin and Wang in \cite{lw11} claimed
that the Triebel-Lizorkin-type space $\dft$ is equivalent to their
generalized Carleson measure space $CMO^{s,q}_{\tau q+1-q/p}(\rn)$
for all $s\in\rr$, $\tau\in[0,\fz)$ and $p,\,q\in(0,\fz)$.
\emph{We denote the index $\alpha$ in \cite{lw11} by $s$ here
as in \cite{yy1,yy2} in accord with the classical Triebel-Lizorkin
spaces when $\tau=0$.} However, in this paper, we first present some
counterexamples to show that this is not true when $\tau\in[0, 1/p)$ (see Proposition \ref{p4} below).
Moreover, by a totally different approach from \cite{lw11}
which may be problematic (see Remark \ref{r3} below),
we prove that for all $p\in(0,\fz]$, $q\in(0,\fz)$ and $\tau\in(1/p,\fz)$, or
$q=\fz$ and $\tau\in[1/p,\fz)$, the Triebel-Lizorkin-type space $\dft$
($p<\fz$) and the Besov-type space $\dbt$ are just the classical Triebel-Lizorkin space
$\dot{F}^{s+n(\tau-1/p)}_{\fz,\fz}(\rn)$ (see Theorem \ref{t1} below),
which further implies that
for all $s\in\rr$, $q\in(0,\fz)$ and $r\in(1,\fz)$,
the generalized Carleson measure space
$CMO^{s,q}_r(\rn)=\dot{F}^{s,r/q}_{q,q}(\rn)=\dot{F}^{s+n(r-1)/q}_{\fz,\fz}(\rn)$
with equivalent norms (see Corollary \ref{c3} below).
As a consequence, we see that for all $s\in\rr$, $p\in(0,\fz)$,
$q\in(0,\fz)$ and $\tau\in(1/p,\fz)$ or $q=\fz$ and $\tau\in[1/p,\fz)$,
$\dft=\dot{F}^{s+n(\tau-1/p)}_{\fz,\fz}(\rn)=CMO^{s,q}_{\tau q+1-q/p}(\rn)$
with equivalent norms; see Corollary \ref{c4}(i) below.
Thus, even in this case, Corollary \ref{c4} also improves
the main results in \cite{lw11}; see Remark \ref{r5} below.
Also, as a direct consequence of the main result (Theorem \ref{t1} below)
of this paper, we know that for all $s\in\rr$ and $p\in(0,\fz]$,
$\dot{B}^{s,1/p}_{p,\fz}(\rn)=\dot{B}^s_{\fz,\fz}(\rn)$ with
equivalent norms, which is sharp in the sense of Remark \ref{r4}
below. Moreover, all results obtained in this paper have inhomogeneous
versions and we only explicitly state the inhomogeneous version of
Theorem \ref{t1} at the end of this paper for similarity;
see Theorem \ref{t2} below. We remark that Theorem \ref{t2} below
answers a question posed in \cite[p.\,168,\,Remark 6.11(i)]{ysy};
see Remark \ref{r6} below.

To recall the notions of $\dbt$ and $\dft$, we need some notation.
Let $\mathcal{S}(\rn)$ be the set of all \emph{Schwartz functions} on $\rn$ endowed
with the usual topology and $\mathcal{S}'(\rn)$ its \emph{topology dual}, namely,
the space of all bounded linear functionals on $\mathcal{S}(\rn)$
endowed with the weak $\ast$-topology.
Following Triebel \cite{t83}, we set
$$\cs_\infty(\rn)\equiv\lf\{\varphi\in\cs(\rn):\ \int_\rn
\varphi(x)x^\gamma\,dx=0\ \mbox{for all multi-indices}\ \gamma\in
\lf(\nn\cup\{0\}\r)^n\r\}$$
and consider $\cs_\fz(\rn)$ as a
subspace of $\cs(\rn)$, including the topology. Use
$\cs'_\infty(\rn)$ to denote the {\it topological dual space} of
$\cs_\infty(\rn)$, namely, the set of all bounded linear
functionals on $\cs_\fz(\rn)$.
Let $\mathcal{P}(\rn)$ be the set of all \emph{polynomials} on $\rn$.
It is well known that
$\cs'_\fz(\rn)=\cs'(\rn)/\mathcal{P}(\rn)$ as topological spaces;
see, for example, \cite[Proposition 8.1]{ysy}.

Let $\mathcal{Q}$ be the set of all \emph{dyadic cubes} in $\rn$, namely,
$$\mathcal{Q}\equiv \{Q_{jk}\equiv 2^{-j}([0,1)^n+k):\ j\in\zz,\ k\in\zz^n\}.$$
For any $Q=Q_{jk}\in\mathcal{Q}$, let $x_Q\equiv 2^{-j}k$, $\ell(Q)$ be
the \emph{side-length} of $Q$, $j_Q\equiv-\log_2\ell(Q)$, and $\chi_Q$ be the
\emph{characteristic function} of $Q$.
For all $j\in\zz$ and $x\in\rn$, Schwartz functions $\vz$ and $Q\in\mathcal{Q}$,
let $\vz_j(x)\equiv 2^{jn}\vz(2^jx)$ and $\vz_Q(x)\equiv |Q|^{-1/2}\vz((x-x_Q)/\ell(Q))$.
Denote by $\widehat{\vz}$
the \emph{Fourier transform} of $\vz\in\mathcal{S}(\rn)$.

Let $\vz\in\mathcal{S}(\rn)$ such that
\begin{equation}\label{e1.1}
\supp \widehat{\vz}\subset \{\xi\in\rn:\,1/2\le|\xi|\le2\}\quad\mathrm{and}\quad
|\widehat{\vz}(\xi)|\ge C>0\hs\mathrm{when}\hs 3/5\le|\xi|\le 5/3.
\end{equation}

We now recall the notions of the Triebel-Lizorkin-type space $\dft$
and the Besov-type space $\dbt$
in \cite{yy1,yy2} as follows.

\begin{defn}\rm\label{d1}
Let $s\in\rr$, $\tau\in[0,\fz)$, $q\in(0,\fz]$ and $\vz\in\cs(\rn)$ satisfy \eqref{e1.1}.

(i) The \emph{Triebel-Lizorkin-type space} $\dft$ with $p\in(0,\fz)$
is defined to be the space of all $f\in \mathcal{S}'_\fz(\rn)$ such that
$$\|f\|_{\dft}\equiv
\sup_{P\in\mathcal{Q}}\frac1{|P|^{\tau}}\left\{\int_P\left[
\sum_{j=j_P}^\fz  2^{js q}|\vz_j\ast f(x)|^q\right]^{p/q}\,dx\right\}^{1/p}<\fz$$
with the usual modification made when $q=\fz$

(ii) The \emph{Besov-type space} $\dbt$ with $p\in(0,\fz]$ is defined to be the space of all $f\in \mathcal{S}'_\fz(\rn)$ such that
$$\|f\|_{\dbt}\equiv
\sup_{P\in\mathcal{Q}}\frac1{|P|^{\tau}}\left\{\sum_{j=j_P}^\fz  2^{js q}\left[\int_P
|\vz_j\ast f(x)|^p\,dx\right]^{q/p}\right\}^{1/q}<\fz$$
with the usual modifications made when $p=\fz$ or $q=\fz$.
\end{defn}

We also recall the Triebel-Lizorkin-Morrey space
$\dot{\mathcal E}_{u,p,q}^s(\rn)$ and the Besov-Morrey space $\dot{\mathcal N}_{u,p,q}^s(\rn)$ introduced in \cite{tx,st} as follows.

\begin{defn}\rm\label{d2}
Let $s\in\rr$, $0<p\le u<\infty,$ $q\in(0,\infty]$ and $\vz\in\cs(\rn)$ satisfy \eqref{e1.1}. The \emph{Triebel-Lizorkin-Morrey space}
$\dot{\mathcal E}_{u,p,q}^s(\rn)$ and the \emph{Besov-Morrey space} $\dot{\mathcal N}_{u,p,q}^s(\rn)$ are defined, respectively,
to be the spaces of all
$f\in\mathcal{S}'_\fz(\rn)$ such that
$$\| f \|_{\dot{\mathcal E}_{u,p,q}^s(\rn)}
\equiv \sup_{P \in {\mathcal Q}} |P|^{\frac{1}{u}-\frac{1}{p}}
\left\{ \int_P \left[ \sum_{j=-\infty}^\infty
2^{jsq}|\varphi_j\ast f(x)|^q \right]^{\frac{p}{q}} \,dx
\right\}^{\frac{1}{p}}<\infty$$
and
$$\| f \|_{\dot{\mathcal N}_{u,p,q}^s(\rn)}
\equiv \left\{\sum_{j=-\infty}^\infty\sup_{P \in {\mathcal Q}} |P|^{\frac{q}{u}-\frac{q}{p}}\left[  \int_P
2^{jsp}|\varphi_j\ast f(x)|^p  \,dx\right]^{\frac{q}{p}}
\right\}^{\frac{1}{q}}<\infty$$
with the usual modifications made when $q=\fz$.
\end{defn}

Some known relations among Triebel-Lizorkin spaces, Besov spaces, Triebel-Lizorkin-type spaces, Besov-type spaces,
Triebel-Lizorkin-Morrey spaces, Besov-Morrey spaces
and $Q$ spaces are summarized as follows. We refer to
\cite{fj90}, \cite[Propositions 3.1 and 3.2]{yy2} and \cite{syy} for more details.

\begin{prop}\label{p1}
Let $s\in\rr$, $\tau\in[0,\fz)$ and $q\in(0,\fz]$. Then

$\mathrm{(i)}$ $\dot{F}^{s,0}_{p,q}(\rn)=\dot{F}^s_{p,q}(\rn)$ for all $p\in(0,\fz)$
and $\dot{B}^{s,0}_{p,q}(\rn)=\dot{B}^s_{p,q}(\rn)$ for all $p\in(0,\fz]$.

$\mathrm{(ii)}$ For all $p\in(0,\fz)$, $\dot{F}^{s,1/p}_{p,q}(\rn)=\dot{F}^s_{\fz,q}(\rn)$
(\cite[Corollary 5.7]{fj90}).

$\mathrm{(iii)}$ For all $p,\,q\in(0,\fz)$,
$\dot{B}^s_{\fz,q}(\rn)$ is a proper subspace of $\dot{B}^{s,1/p}_{p,q}(\rn)$;
for all $q\in(0,\fz)$, $\dot{B}^{s,1/p}_{p,q}(\rn)\subseteq
\dot{B}^{s,1/q}_{q,q}(\rn) $ if $p\ge q$ and
$\dot{B}^{s,1/q}_{q,q}(\rn)\subseteq
\dot{B}^{s,1/p}_{p,q}(\rn) $ if $p\le q$
(\cite[Proposition 3.2]{yy2}).

$\mathrm{(iv)}$ If $\tau<0$, then
$\dot{F}^{s,\tau}_{p,q}(\rn)=\dot{B}^{s,\tau}_{p,q}(\rn)=\cp(\rn)$.

$\mathrm{(v)}$ $Q_\az(\rn)=\dot{F}^{\az,1/2-\az/n}_{2,2}(\rn)$ for all $\az\in(0,\min\{1, n/2\})$ (\cite[Corollary 3.1]{yy1}).

$\mathrm{(vi)}$ For all $0<p\le u<\fz$ and $q\in(0,\fz]$,
$\dot{\mathcal{E}}^s_{u,p,q}(\rn)=\dot{F}^{s,1/p-1/u}_{p,q}(\rn)$ and $\dot{\mathcal{N}}^s_{u,p,\fz}(\rn)=\dot{B}^{s,1/p-1/u}_{p,\fz}(\rn)$
with equivalent norms;
for all $0<p<u<\fz$ and $q\in(0,\fz)$, $\dot{\mathcal{N}}^s_{u,p,q}(\rn)\subsetneqq\dot{B}^{s,1/p-1/u}_{p,q}(\rn)$ (\cite[Theorem 1.1]{syy}).
\end{prop}

The corresponding sequence spaces, $\dsft$ and $\dsbt$, of the spaces $\dft$
and $\dbt$, were also introduced in \cite{yy2}.

\begin{defn}\rm\label{d3} Let $s\in\rr$, $\tau\in[0,\fz)$ and $q\in(0,\fz]$.
The \emph{sequences spaces} $\dsft$ with $p\in(0,\fz)$ and $\dsbt$ with $p\in(0,\fz]$
are defined, respectively, to be the space of all sequences
$t\equiv\{t_Q\}_{Q\in\mathcal{Q}}\subset\cc$ such that $\|t\|_{\dsft}<\fz$
and $\|t\|_{\dsbt}<\fz$, where
$$\|t\|_{\dsft}\equiv
\sup_{P\in\mathcal{Q}}\frac1{|P|^{\tau}}\left\{\int_P\left(
\sum_{Q\subset P} \left[|Q|^{-s/n-1/2}|t_Q|\chi_Q(x)\right]^q\right)^{p/q}\,dx\right\}^{1/p}$$
and
$$\|t\|_{\dsbt}\equiv
\sup_{P\in\mathcal{Q}}\frac1{|P|^{\tau}}\left\{\sum_{j=j_P}^\fz\left(\int_P
\left[\sum_{\gfz{Q\subset P}{\ell(Q)=2^{-j}}} |Q|^{-s/n-1/2}|t_Q|\chi_Q(x)\right]^p\,dx\right)^{q/p}\right\}^{1/q}$$
with the usual modifications made when $p=\fz$ or $q=\fz$.
\end{defn}

Via the Calder\'on reproducing formula, we in \cite{yy2}
established the $\vz$-transform characterizations of the
spaces $\dft$ and $\dbt$, which implies the following conclusions.

\begin{prop}\label{p2} Let $s\in\rr$, $\tau\in[0,\fz)$, $q\in(0,\fz]$
and $\vz\in\cs(\rn)$ satisfy \eqref{e1.1}.

$\mathrm{(i)}$ For all $p\in(0,\fz)$,
$f\in\dft$ if and only if $f\in\cs'_\fz(\rn)$ and
$$\sup_{P\in\mathcal{Q}}\frac1{|P|^{\tau}}\left\{\int_P\left(
\sum_{Q\subset P} \left[|Q|^{-s/n-1/2}|\langle f,\vz_Q\rangle|\chi_Q(x)\right]^q\right)^{\frac pq}\,dx\right\}^{\frac 1p}<\fz.$$
Moreover, $\|f\|_{\dft}$ is equivalent to
$\|\{\langle f,\vz_Q\rangle\}_{Q\in\cq}\|_{\dsft}$
with equivalent constants independent of $f$.

$\mathrm{(ii)}$ For all $p\in(0,\fz]$, $f\in\dbt$ if and only if $f\in\cs'_\fz(\rn)$ and
$$\sup_{P\in\mathcal{Q}}\frac1{|P|^{\tau}}\left\{\sum_{j=j_P}^\fz\left(\int_P
\left[\sum_{\gfz{Q\subset P}{\ell(Q)=2^{-j}}} |Q|^{-s/n-1/2}|\langle f,\vz_Q\rangle|\chi_Q(x)
\right]^p\,dx\right)^{\frac qp}\right\}^{\frac 1q}<\fz.$$
Moreover, $\|f\|_{\dbt}$ is equivalent to
$\|\{\langle f,\vz_Q\rangle\}_{Q\in\cq}\|_{\dsbt}$
with equivalent constants independent of $f$.
\end{prop}

\begin{rem}\rm\label{r1}
In lines 7 through 11 of \cite[p.\,921]{lw11}, Lin and Wang
said that the Triebel-Lizorkin-type space $\dft$
with $s,\,\tau\in\rr$, $p\in(1,\fz)$ and $q\in (1,\fz]$
was defined in \cite{yy1} as the space of all $f\in\cs'_\fz(\rn)=\cs'(\rn)/\mathcal{P}(\rn)$
such that
$$\|f\|_{\dft}\equiv
\lf\|\lf\{\langle f,\vz_Q\rangle\r\}_{Q\in\cq}\r\|_{\dsft}<\fz.$$
However, these spaces in \cite{yy1} were defined
as in Definition \ref{d1}. Moreover, since we did not establish
the $\vz$-transform characterization of these spaces in \cite{yy1},
we did not introduce the space $\dsft$ of sequences in \cite{yy1}.
Thus, Proposition \ref{p2}(i)
is not included in \cite{yy1}.
However, we do deduce Proposition \ref{p2}
from the $\vz$-transform characterizations of $\dft$ and $\dbt$
obtained in a later paper \cite{yy2}.
\end{rem}

The generalized Carleson measure space $CMO^{s,q}_r(\rn)$ for $s,\,r\in\rr$ and $q\in(0,\fz]$
and the space $\dot{B}BMO^{s,q}_p(\rn)$ for $s\in\rr$ and $p,\,q\in(0,\fz]$ were introduced, respectively, by Lin and Wang
in \cite{lw11} and \cite{lw0}.

\begin{defn}\rm\label{d4} Let $s\in\rr$ and $q\in(0,\fz]$.

(i) If $r\in\rr$, then the \emph{generalized Carleson measure space} $CMO^{s,q}_r(\rn)$ is defined to be the space of all
$f\in\cs'_\fz(\rn)$ such that
$$\|f\|_{CMO^{s,q}_r(\rn)}\equiv
\sup_{P\in\mathcal{Q}}\left\{|P|^{-r}\int_P\sum_{Q\subset P}
\left [|Q|^{-s/n-1/2}|\langle f,\vz_Q\rangle|\chi_Q(x)\right]^q\,dx\right\}^{1/q}<\fz$$
with the usual modification made when $q=\fz$.

(ii) If $p\in(0,\fz]$, the space $\dot{B}BMO^{s,q}_p(\rn)$ is defined to be the
space of all $f\in\cs'_\fz(\rn)$ such that
$\|f\|_{\dot{B}BMO^{s,q}_p(\rn)}<\fz$, where
$$\|f\|_{\dot{B}BMO^{s,q}_p(\rn)}
\equiv\sup_{P\in\mathcal{Q}}\left\{\sum_{v=j_P}^\fz
\left[\frac1{|P|}\sum_{\gfz{Q\subset P}{\ell(Q)=2^{-v}}}\left(|Q|^{-s/n-1/2+1/p}|\langle f, \vz_Q\rangle|\right)^p\right]^{q/p}\right\}^{1/q}$$
with the usual modifications made when $p=\fz$ or $q=\fz$.
\end{defn}

\begin{rem}\rm\label{r2}
(i) In \cite{lw11}, Lin and Wang claimed that the generalized Carleson measure space $CMO^{s,q}_r(\rn)$ was first introduced by themselves
in \cite{lw}.

(ii) As was mentioned in \cite{yy2}, the space $\dot{B}BMO^{s,q}_p(\rn)$ was introduced in \cite{lw0} which was the only preprint we had
from Lin and Wang when our paper \cite{yy2} was being written.
In \cite[p.\,463]{yy2}, we even showed that $\dot{B}BMO^{s,q}_p(\rn)$ is a special case of Besov-type spaces $\dbt$. The spaces $\dot{B}BMO^{s,q}_p(\rn)$
and $CMO^{s,q}_r(\rn)$ do obviously not coincide; see Proposition \ref{p3}
below.
\end{rem}

From Propositions \ref{p1} and \ref{p2}, and Definition \ref{d4},
it is easy to deduce that the spaces $CMO^{s,q}_r(\rn)$ and
$\dot{B}BMO^{s,q}_p(\rn)$ are, respectively, special
cases of Triebel-Lizorkin-type spaces, Triebel-Lizorkin-Morrey
spaces and Besov-type spaces as follows;
see also \cite[p.\,463]{yy2} and \cite[p.\,921]{lw11}.

\begin{prop}\label{p3} Let $s\in\rr$ and $q\in(0,\fz]$.

{\rm(i)} For all $r\in[0,\fz)$ and $q\in(0,\fz]$,
$\dot{F}^{s,r/q}_{q,q}(\rn)=CMO^{s,q}_r(\rn)$
with equivalent norms. In particular, when $r\in(0,1)$, $\dot{\mathcal{E}}^s_{\frac{q}{1-r},\,q,\,q}(\rn)=CMO^{s,q}_r(\rn)$.

{\rm(ii)} For all $p\in(0,\fz]$,
$\dot{B}^{s,1/p}_{p,q}(\rn)=\dot{B}BMO^{s,q}_p(\rn)$
with equivalent norms.
\end{prop}

The following is just \cite[Theorem 1]{lw11} with $\az$ replaced by $s$,
which is the main result of \cite{lw11}.

\setcounter{thm}{0}
\renewcommand{\thethm}{\Alph{thm}}

\begin{thm}\label{ta}
Let $s,\,\tau\in\mathbb{R}$ and $p,\,q\in(0,\infty)$.
Then
$$\|\{\laz f,\vz_Q\raz\}_{Q\in\cq}\|_{\dsft}
\sim \|\{\laz f,\vz_Q\raz\}_{Q\in\cq}\|_{\dot{f}^{s,\tau+1/q-1/p}_{q,q}(\rn)}.$$
\end{thm}

The following corollary is immediately deduced from Theorem \ref{ta},
which is just \cite[Corollary 6]{lw11}
with $\az$ replaced by $s$.

\setcounter{cor}{0}
\renewcommand{\thecor}{\Alph{thm}}

\begin{cor}\label{ca}
Let $s,\,\tau\in\mathbb{R}$ and $p,\,q\in(0,\infty)$.
Then $\|f\|_{\dft}
\sim \|f\|_{CMO^{s,q}_{\tau q+1-q/p}(\rn)}.$
\end{cor}

Moreover, Theorem \ref{ta} is a direct consequence of
the following Theorem \ref{tb}, which is \cite[Theorem 2]{lw11}
with $\az$ replaced by $s$.

\begin{thm}\label{tb}
Let $s,\,\tau\in\mathbb{R}$ and $p,\,q\in(0,\infty)$.
Then $\|t\|_{\dsft}
\sim \|t\|_{\dot{f}^{s,\tau+1/q-1/p}_{q,q}(\rn)}.$
\end{thm}

Indeed, Theorems \ref{ta} and \ref{tb} and Corollary \ref{ca}
when $p=q$ are obvious,
and when $\tau=1/p$ are just \cite[Corollary 5.7]{fj90}; see
also Proposition \ref{p1}(ii).
However, it seems that Theorems \ref{ta} and \ref{tb} and
Corollary \ref{ca} may be not true for some parameters.

\begin{rem}\rm\label{r3}
It seems that there exist two gaps in the proof of Theorem \ref{tb}
in \cite{lw11}. For the convenience of the reader,
in this remark, we use the same notation as in pages 922 and 923,
and page 925 of \cite{lw11}.

First, as in \cite[p.\,922]{lw11}, for all $\az\in\rr$, dyadic cubes $P$,
sequences $\textbf{s}=\{s_Q\}_{Q\in\mathcal{Q}}$ and $x\in\rn$, let
$$G_P^{\az,\tau,q}(\textbf{s})(x)\equiv |P|^{-\tau+1/q}\left\{
\sum_{Q\subset P} \left[|Q|^{-\az/n-1/2}|s_Q|\chi_Q(x)\right]^q\right\}^{1/q},$$
$$m^{\az,\tau,q}(\textbf{s})(x)\equiv\sup_{P\in\mathcal{Q}}\inf\left\{\varepsilon:\
|\{x\in P:\ G_P^{\az,\tau,q}(\textbf{s})(x)> \varepsilon\}|<|P|/4\right\}$$
and
$$v(x)\equiv\inf\{j\in\zz:\ G_P^{\az,\tau,q}(\textbf{s})(x)\le
m^{\az,\tau,q}(\textbf{s})(x),\ \ell(P)=2^{-j}\}.$$

The definition of $v(x)$ is problematic.
It seems that the infimum should be taken over all
dyadic cubes $P$ containing $x$; otherwise $v(x)\equiv -\fz$.

Even if this change is made, it is not clear whether
$G_P^{\az,\tau,q}(\textbf{s})(x)$ is monotonic on $P$,
that is, $G_{P_1}^{\az,\tau,q}(\textbf{s})(x)
\le G_{P_2}^{\az,\tau,q}(\textbf{s})(x)$ when $P_1\subset P_2$
for all $x\in\rn$.
Then the first equality \cite[p.\,923]{lw11}, namely,
$$E_Q\equiv\{x\in Q : 2^{-v(x)}\ge \ell(Q)\}=\{x\in Q: G_{Q}^{\az,\tau,q}(\textbf{s})(x)
\le m^{\az,\tau,q}(\textbf{s})(x)\},$$
is problematic. More precisely, the embedding
$$\{x\in Q : 2^{-v(x)}\ge \ell(Q)\}\subset\{x\in Q: G_{Q}^{\az,\tau,q}(\textbf{s})(x)
\le m^{\az,\tau,q}(\textbf{s})(x)\}$$
may be not true. So, all the proofs break down here.

Second, for all
sequences $\textbf{t}=\{t_Q\}_Q\in \dot{F}^{\az,\tau}_{p,q}$,
let $\mathcal{Q}(\textbf{t})$ be the collection of all dyadic cubes $Q$ so that
$t_Q\neq0$ and enumerated as $\mathcal{Q}(\textbf{t})\equiv\{P_1,P_2,P_3,\cdots\}$.
It was claimed in \cite[p.\,925]{lw11} that
$\textbf{t}_m$ converges to $\textbf{t}$ in $\dot{F}^{\az,\tau}_{p,q}$
as $m\to\infty$, where $\textbf{t}_m$ is a sequence
containing $n$ non-zero elements of $\textbf{t}$, namely,
$\textbf{t}_m\equiv\{(t_m)_Q\}_{Q\in\cq}$ is defined by setting $(t_m)_Q\equiv t_Q$
if $Q\in\{P_1,\cdots,P_m\}$, otherwise $(t_m)_Q\equiv0$
(\emph{We replace $n$ in \cite[p.\,925]{lw11} by $m$ here to
distinguish the dimension of $\rn$}).
However, this may also not be true
when $\tau>0$. For example, when $\az=0$, $p=q=2$ and $\tau=1/2$,
then $\dot{f}^{0,1/2}_{2,2}$ ($=\dot{f}^0_{\fz,2}$)
is the corresponding sequence space of
$\bmo$. If $\textbf{t}_m\to \textbf{t}$
in $\dot{f}^{0,1/2}_{2,2}$, applying the $\vz$-transform characterization of $\bmo$, we see that $\mathcal{S}_\infty(\rn)$ is dense
in $\bmo$. But, it is well known that this is not the case.
\end{rem}

Indeed, Theorems \ref{ta} and \ref{tb} are not true when $\tau\in[0,1/p)$.
To see this, let $\tau\in[0,1/p)$ and $q\in (p,\infty)$ such that
$\tau+1/q-1/p<0$. Then by Proposition \ref{p1}(iv), the space $\dot{F}^{s,\tau+1/q-1/p}_{q,q}(\rn)=\mathcal{P}(\rn)$.
However, it was proved in \cite[Proposition 3.1]{yy2} that
the space
$\dft$ when $\tau\in[0,\fz)$ contains $\mathcal{S}_\infty(\rn)$,
which is a contradiction.

The following proposition give a more concrete
counterexample to Theorem \ref{tb}.
Recall that $\dsft=\dsbt$ if $p=q\in(0,\fz)$.

\begin{prop}\label{p4}
Let $s\in\rr$.

$\mathrm{(i)}$ For all $p\in(0,\fz)$, if $q\in (p,\fz)$ and $\tau \in (0,1/p-1/q]$, or
$q=\fz$ and $\tau \in (0,1/p-1/q)$,
the space $\dot{b}^{s,\tau+1/q-1/p}_{q,q}(\rn)$ is a proper subspace of $\dsft$.

$\mathrm{(ii)}$ For all $p\in(0,\fz)$, if $q\in (p,\fz)$ and $\tau \in (0,1/p-1/q]$, or
$q=\fz$ and $\tau \in [0,1/p-1/q)$,
the space $\dot{b}^{s,\tau+1/q-1/p}_{q,q}(\rn)$ is a proper subspace of $\dsbt$.
\end{prop}

\begin{proof}
(i) The embedding
$\dot{b}^{s,\tau+1/q-1/p}_{q,q}(\rn)\subset \dsft$ is a direct
consequence of H\"older's inequality.
We only show that these two spaces are not equivalent
in the case that $q\in (p,\fz)$. The proof of the case that
$q=\fz$ is similar and we omit the details.

To this end, for all $j\in\zz$, let $R_j\equiv[0, 2^{-j})^n$. Define $t\equiv\{t_Q\}_Q$ by setting
$t_Q\equiv|R_j|^{s/n+1/2+\tau-1/p}$ when $Q=R_j$ for some $j\in\zz$, otherwise
$t_Q\equiv0$. Then, by $\tau>0$, we conclude that
\begin{eqnarray*}
\|t\|_{\dsft}
&&=\sup_{P\in\mathcal{Q}}\frac1{|P|^{\tau}}\left\{\int_P
\left(\sum_{Q\subset P} \left[|Q|^{-s/n-1/2}|t_Q|\chi_Q(x)\right]^q\right)^{p/q}\,dx\right\}^{1/p}\\
&&=\sup_{k\in\zz}|R_k|^{-\tau}\left\{\int_{R_k}
\left(\sum_{j=k}^\fz  \left[|R_j|^{-s/n-1/2}|t_{R_j}|\chi_{R_j}(x)
\right]^q\right)^{p/q}\,dx\right\}^{1/p}\\
&&\le\sup_{k\in\zz}|R_k|^{-\tau}\left\{
\sum_{j=k}^\fz |R_j|^{(\tau-1/p)p}|R_j|\right\}^{1/p}\sim1,
\end{eqnarray*}
while from $\tau\le 1/p-1/q$, it follows that
\begin{eqnarray*}
\|t\|_{\dot{b}^{s,\tau+1/q-1/p}_{q,q}(\rn)}
&&=\sup_{P\in\mathcal{Q}}|P|^{-(\tau+1/q-1/p)}\left\{\int_P
\sum_{Q\subset P} \left[|Q|^{-s/n-1/2}|t_Q|\chi_Q(x)\right]^q\,dx\right\}^{1/q}\\
&&=\sup_{k\in\zz}|R_k|^{-(\tau+1/q-1/p)}\left\{\int_{R_k}
\sum_{j=k}^\fz|R_j|^{(\tau-1/p)q}\chi_{R_j}(x)\,dx\right\}^{1/q}\\
&&=\sup_{k\in\zz}|R_k|^{-(\tau+1/q-1/p)}\left\{
\sum_{j=k}^\fz |R_j|^{(\tau-1/p)q}|R_j|\right\}^{1/q}=\infty.
\end{eqnarray*}
Thus, $\|t\|_{\dsft}$ and $\|t\|_{\dot{b}^{s,\tau+1/q-1/p}_{q,q}(\rn)}$
are not equivalent, which implies that
$$\dot{b}^{s,\tau+1/q-1/p}_{q,q}(\rn)\subsetneqq\dsft.$$

(ii) Similarly, by H\"older's inequality, we see that $\dot{b}^{s,\tau+1/q-1/p}_{q,q}(\rn)\subset \dsbt$.
Again, we only show that $\dot{b}^{s,\tau+1/q-1/p}_{q,q}(\rn)\subsetneqq\dsbt$
in the case $q\in (p,\fz)$. The proof of the case that $q=\fz$ is
similar and we omit the details.

Let $t$ be as in the proof of (i). Then from $\tau\le 1/p-1/q$, we
infer that $\|t\|_{\dot{b}^{s,\tau+1/q-1/p}_{q,q}(\rn)}=\fz$. However,
by $\tau>0$, we obtain
\begin{eqnarray*}
\|t\|_{\dsbt}
&&=\sup_{P\in\mathcal{Q}}\frac1{|P|^{\tau}}\left\{\sum_{j=j_P}^\fz
\left(\int_P \sum_{\gfz{Q\subset P}{\ell(Q)=2^{-j}}} \left[|Q|^{-s/n-1/2}|t_Q|\chi_Q(x)\right]^p\,dx\right)^{q/p}\right\}^{1/q}\\
&&=\sup_{k\in\zz}|R_k|^{-\tau}\left\{\sum_{j=k}^\fz
\left(\int_{R_k} \left[|R_j|^{-s/n-1/2}|t_{R_j}|\chi_{R_j}(x)
\right]^p\,dx\right)^{q/p}\right\}^{1/q}\\
&&\le\sup_{k\in\zz}|R_k|^{-\tau}\left\{
\sum_{j=k}^\fz |R_j|^{\tau q}\right\}^{1/q}\sim1,
\end{eqnarray*}
which completes the proof of Proposition \ref{p4}.
\end{proof}

When $\tau\in(1/p,\fz)$, we use a totally
different approach from the proof of \cite[Theorem 2]{lw11}
to obtain the following conclusions,
which have independently interest and may be useful in applications.

\setcounter{thm}{0}
\renewcommand{\thethm}{\arabic{thm}}

\begin{thm}\label{t1}
Let $s\in\rr$, $q\in (0,\fz]$.

$\mathrm{(i)}$ For all $p\in(0,\fz)$, $q\in(0,\fz)$ and $\tau\in(1/p,\fz)$, or
$q=\fz$ and $\tau\in[1/p,\fz)$,
$$\dft=\dot{F}^{s+n(\tau-1/p)}_{\fz,\fz}(\rn)$$
with equivalent norms.

$\mathrm{(ii)}$ For all $p\in(0,\fz]$, $q\in(0,\fz)$ and $\tau\in(1/p,\fz)$, or
$q=\fz$ and $\tau\in[1/p,\fz)$,
$$\dbt=\dot{B}^{s+n(\tau-1/p)}_{\fz,\fz}(\rn)$$
with equivalent norms.
\end{thm}

\begin{proof}
(i) By the $\vz$-transform characterizations of the spaces $\dft$
in \cite{yy2} and the space $\dot{F}^s_{p,q}(\rn)$ in \cite{fj90},
to prove (i), it suffices to show that
$\dsft=\dot{f}^{s+n(\tau-1/p)}_{\fz,\fz}(\rn)$
with equivalent norms, where $\dot{f}^{s+n(\tau-1/p)}_{\fz,\fz}(\rn)$
is the sequence space of $\dot{F}^{s+n(\tau-1/p)}_{\fz,\fz}(\rn)$;
see \cite{fj90}.

To see $\|\cdot\|_{\dsft}\sim\|\cdot\|_{\dot{f}^{s+n(\tau-1/p)}_{\fz,\fz}(\rn)}$,
we recall that for all $t\equiv\{t_Q\}_{Q\in\cq}$,
$$\|t\|_{\dsft}\equiv
\sup_{P\in\mathcal{Q}}\frac1{|P|^{\tau}}\left\{\int_P\left(
\sum_{Q\subset P} \left[|Q|^{-s/n-1/2}|t_Q|\chi_Q(x)\right]^q\right)^{p/q}\,dx\right\}^{1/p}$$
and
$$\|t\|_{\dot{f}^{s+n(\tau-1/p)}_{\fz,\fz}(\rn)}\equiv
\sup_{Q\in\mathcal{Q}}|Q|^{-s/n-(\tau-1/p)-1/2}|t_Q|.$$
Obviously, we have
$\|t\|_{\dot{f}^{s+n(\tau-1/p)}_{\fz,\fz}(\rn)}\le \|t\|_{\dsft}$
for all sequences $t\equiv\{t_Q\}_{Q\in\mathcal{Q}}$.
On the other hand, by the assumption on $\tau$, we conclude that
\begin{eqnarray*}
\|t\|_{\dsft}&&=
\sup_{P\in\mathcal{Q}}\frac1{|P|^{\tau}}\left\{\int_P\left(
\sum_{j=j_P}^\fz\sum_{\gfz{Q\subset P}{\ell(Q)=2^{-j}}} \left[|Q|^{-s/n-1/2}|t_Q|\chi_Q(x)\right]^q\right)^{p/q}\,dx\right\}^{1/p}\\
&&\le \|t\|_{\dot{f}^{s+n(\tau-1/p)}_{\fz,\fz}(\rn)}\\
&&\hs\times\sup_{P\in\mathcal{Q}}\frac1{|P|^{\tau}}\left\{\int_P\left[
\sum_{j=j_P}^\fz\sum_{\gfz{Q\subset P}{\ell(Q)=2^{-j}}} |Q|^{(\tau-1/p)q}\chi_Q(x)\right]^{p/q}\,dx\right\}^{1/p}\\
&&=\|t\|_{\dot{f}^{s+n(\tau-1/p)}_{\fz,\fz}(\rn)}\\
&&\hs\times\sup_{P\in\mathcal{Q}}\frac1{|P|^{\tau}}\left\{\int_P\left[
\sum_{j=j_P}^\fz 2^{-jn(\tau-1/p)q}\chi_P(x)\right]^{p/q}\,dx\right\}^{1/p}\\
&&\ls \|t\|_{\dot{f}^{s+n(\tau-1/p)}_{\fz,\fz}(\rn)}
\sup_{P\in\mathcal{Q}}\frac1{|P|^{\tau}}\left\{\int_P |P|^{(\tau-1/p)p}
\,dx\right\}^{1/p}\sim \|t\|_{\dot{f}^{s+n(\tau-1/p)}_{\fz,\fz}(\rn)},
\end{eqnarray*}
which further implies that
$\dft=\dot{F}^{s+n(\tau-1/p)}_{\fz,\fz}(\rn)$ with equivalent norms
and completes the proof of (i).

(ii) Similarly, we only need to
show that $\dsbt=\dot{b}^{s+n(\tau-1/p)}_{\fz,\fz}(\rn).$
The inequality
$\|\cdot\|_{\dot{b}^{s+n(\tau-1/p)}_{\fz,\fz}(\rn)}\le \|\cdot\|_{\dsbt}$
is trivial. On the other hand, by the assumption on $\tau$, we see that
\begin{eqnarray*}
\|t\|_{\dsbt}&&=
\sup_{P\in\mathcal{Q}}\frac1{|P|^{\tau}}\left\{\sum_{j=j_P}^\fz\left(\int_P
\sum_{\gfz{Q\subset P}{\ell(Q)=2^{-j}}} \left[|Q|^{-s/n-1/2}|t_Q|\chi_Q(x)\right]^p\,dx\right)^{q/p}\right\}^{1/q}\\
&&\le \|t\|_{\dot{b}^{s+n(\tau-1/p)}_{\fz,\fz}(\rn)}\\
&&\hs\times\sup_{P\in\mathcal{Q}}\frac1{|P|^{\tau}}\left\{\sum_{j=j_P}^\fz\left[\int_P
\sum_{\gfz{Q\subset P}{\ell(Q)=2^{-j}}} |Q|^{(\tau-1/p)p}\chi_Q(x)\,dx\right]^{q/p}\right\}^{1/q}\\
&&=\|t\|_{\dot{b}^{s+n(\tau-1/p)}_{\fz,\fz}(\rn)}\\
&&\hs\times\sup_{P\in\mathcal{Q}}\frac1{|P|^{\tau}}\left\{\sum_{j=j_P}^\fz 2^{-jn(\tau-1/p)q}\left[\int_P
\chi_P(x)\,dx\right]^{q/p}\right\}^{1/q}
\sim \|t\|_{\dot{b}^{s+n(\tau-1/p)}_{\fz,\fz}(\rn)},
\end{eqnarray*}
which completes the proof of Theorem \ref{t1}.
\end{proof}

Observe that $\tau+1/q-1/p>1/q$ when $\tau\in(1/p,\fz)$.
As a direct consequence of Theorem \ref{t1},
we have the following conclusions, comparing with \cite[Theorem 1]{lw11}
(see also Theorem \ref{ta}).

\setcounter{cor}{0}
\renewcommand{\thecor}{\arabic{cor}}

\begin{cor}\label{c1}
Let $s\in\rr$, $q\in (0,\fz]$.

$\mathrm{(i)}$ For all $p\in(0,\fz)$, $q\in(0,\fz)$ and $\tau\in(1/p,\fz)$, or
$q=\fz$ and $\tau\in[1/p,\fz)$,
$$\dft=\dot{F}^{s+n(\tau-1/p)}_{\fz,\fz}(\rn)
=\dot{B}^{s,\tau+1/q-1/p}_{q,q}(\rn)
\lf(=\dot{F}^{s,\tau+1/q-1/p}_{q,q}(\rn)\ \mathrm{if}\
q<\fz\r)$$
with equivalent norms.

$\mathrm{(ii)}$ For all $p\in(0,\fz]$, $q\in(0,\fz)$ and $\tau\in(1/p,\fz)$, or
$q=\fz$ and $\tau\in[1/p,\fz)$,
$$\dbt=\dot{B}^{s+n(\tau-1/p)}_{\fz,\fz}(\rn)
=\dot{B}^{s,\tau+1/q-1/p}_{q,q}(\rn)$$
with equivalent norms.
\end{cor}

Another special case of Theorem \ref{t1}
is the following conclusion, which has independently interest, comparing with Proposition \ref{p1}(ii).

\begin{cor}\label{c2}
Let $s\in\rr$ and $p\in(0,\fz]$. Then $\dot{B}^{s,1/p}_{p,\fz}(\rn)
=\dot{B}^s_{\fz,\fz}(\rn)$ with equivalent norms.
\end{cor}

\begin{rem}\rm\label{r4}
We remark that Corollary \ref{c2} is sharp in the following sense:
for all $s\in\rr$ and $p,\,q\in(0,\fz)$, $\dot{B}^s_{\fz,q}(\rn)
\subsetneqq \dot{B}^{s,1/p}_{p,q}(\rn)$
by Proposition \ref{p1}(iii). This is totally different from
the case of the Triebel-Lizorkin spaces (see Proposition \ref{p1}(ii) again).
\end{rem}

As a direct consequence of Proposition \ref{p3} and
Theorem \ref{t1}, we deduce that
when $r\in(1,\fz)$, the space $CMO^{s,q}_r(\rn)$ is essentially the
Triebel-Lizorkin space.

\begin{cor}\label{c3}
Let $s\in\rr$. If $q\in(0,\fz)$ and $r\in (1,\fz)$ or $q=\fz$
and $r\in [1,\fz)$, then
$$CMO^{s,q}_r(\rn)=\dot{B}^{s,r/q}_{q,q}(\rn)
=\dot{F}^{s+n(r-1)/q}_{\fz,\fz}(\rn)\left(=\dot{F}^{s,r/q}_{q,q}(\rn)\ \mathrm{if}\ q<\fz\r)$$
with equivalent norms.
\end{cor}

This corollary can be re-written as follows, which implies that
the conclusions of Theorems \ref{ta} and \ref{tb}, and Corollary \ref{ca} are correct when $\tau\in(1/p,\fz)$.

\begin{cor}\label{c4}
Let $s\in\rr$.

{\rm(i)} If $p\in(0,\fz)$, $q\in (0,\fz)$ and $\tau\in(1/p,\fz)$
or $q=\fz$ and $\tau\in [1/p,\fz)$, then
$$\dft=\dot{F}^{s+n(\tau-1/p)}_{\fz,\fz}(\rn)=CMO^{s,q}_{\tau q+1-q/p}(\rn)$$
with equivalent norms.

{\rm(ii)} If $p\in(0,\fz]$, $q\in (0,\fz)$ and $\tau\in(1/p,\fz)$
or $q=\fz$ and $\tau\in [1/p,\fz)$, then
$$\dbt=\dot{B}^{s+n(\tau-1/p)}_{\fz,\fz}(\rn)=CMO^{s,q}_{\tau q+1-q/p}(\rn)$$
with equivalent norms.
\end{cor}

\begin{rem}\rm\label{r5} Although the assertion
that for all $s\in\rr$, $p,\,q\in(0,\fz)$
and $\tau\in(1/p,\fz)$,
$\dft=CMO^{s,q}_{\tau q+1-q/p}(\rn)$ with equivalent norms, was claimed in
\cite[Corollary 6]{lw11}, its proof therein is problematic;
see Remark \ref{r3}.
\end{rem}

We point out that Theorem \ref{t1}
is also true for inhomogeneous Triebel-Lizorkin-type spaces
$F^{s,\tau}_{p,q}(\rn)$ and Besov-type
spaces $B^{s,\tau}_{p,q}(\rn)$ introduced in \cite{ysy}.
Let $\Phi\in\cs(\rn)$
such that
\begin{equation}\label{e1.2}
\supp \widehat{\Phi}\subset \{\xi\in\rn:\ |\xi|\le2\}\quad \mathrm{and}\quad |\widehat{\Phi}(\xi)|\ge C>0\hs \mathrm{if}\hs |\xi|\le 5/3.
\end{equation}
The inhomogeneous Triebel-Lizorkin-type space $F^{s,\tau}_{p,q}(\rn)$
and the inhomogeneous Besov-type space $B^{s,\tau}_{p,q}(\rn)$
in \cite{ysy} were defined as follows.

\begin{defn}\rm\label{d5} Let $s\in\rr$, $\tau\in[0,\fz)$ and $q\in(0,\fz]$.
Let $\vz$ be as in \eqref{e1.1} and $\vz_0\equiv\Phi$ be as in \eqref{e1.2}.

(i) If $p\in(0,\fz)$, the \emph{inhomogeneous Triebel-Lizorkin-type space}
$F^{s,\tau}_{p,q}(\rn)$ is defined
to be the space of all $f\in \mathcal{S}'(\rn)$ such that
$$\|f\|_{F^{s,\tau}_{p,q}(\rn)}\equiv
\sup_{P\in\mathcal{Q}}\frac1{|P|^{\tau}}\left\{\int_P\left[
\sum_{j=\max\{0,\,j_P\}}^\fz  2^{js q}|\vz_j\ast f(x)|^q\right]^{p/q}\,dx\right\}^{1/p}<\fz$$
with the usual modification made when $q=\fz$.

(ii) If $p\in(0,\fz]$, the \emph{inhomogeneous Besov-type space} $B^{s,\tau}_{p,q}(\rn)$ is defined to be the space
of all $f\in \mathcal{S}'(\rn)$ such that
$$\|f\|_{B^{s,\tau}_{p,q}(\rn)}\equiv
\sup_{P\in\mathcal{Q}}\frac1{|P|^{\tau}}\left\{\sum_{j=\max\{0,\,j_P\}}^\fz  2^{js q}\left[\int_P
|\vz_j\ast f(x)|^p\,dx\right]^{q/p}\right\}^{1/q}<\fz$$
with the usual modifications made when $p=\fz$ or $q=\fz$.
\end{defn}

All conclusions of Propositions \ref{p1}
through \ref{p4} have inhomogeneous versions
and we omit the details.
Moreover, we have the following conclusions, whose proofs are similar to that of Theorem \ref{t1}. We also omit the details.

\begin{thm}\label{t2}
Let $s\in\rr$ and $q\in (0,\fz]$.

$\mathrm{(i)}$ For all $p\in(0,\fz)$, $q\in(0,\fz)$ and $\tau\in(1/p,\fz)$, or
$q=\fz$ and $\tau\in[1/p,\fz)$,
$$F^{s,\tau}_{p,q}(\rn)=F^{s+n(\tau-1/p)}_{\fz,\fz}(\rn)$$
with equivalent norms.

$\mathrm{(ii)}$ For all $p\in(0,\fz]$, $q\in(0,\fz)$ and $\tau\in(1/p,\fz)$, or
$q=\fz$ and $\tau\in[1/p,\fz)$,
$$B^{s,\tau}_{p,q}(\rn)=B^{s+n(\tau-1/p)}_{\fz,\fz}(\rn)$$
with equivalent norms.
\end{thm}

\begin{rem}\rm\label{r6} It was asked in \cite[p.\,168,\,Remark 6.11(i)]{ysy}
that for which set of parameters $p,\,q,\,\tau$, the spaces
$F^{s,\tau}_{p,q}(\rn)$ and $B^{s,\tau}_{p,q}(\rn)$ coincide
with the H\"older-Zygmund spaces? Some special cases were obtained in
\cite[p.\,167,\,Theorem 6.9]{ysy}. Obviously, Theorem \ref{t2}
above gives a complete answer to this question.
\end{rem}

From Theorem \ref{t2}, we also deduce inhomogeneous versions
of all conclusions in Corollaries \ref{c1} through \ref{c4}.
We omit the details again.

\medskip

\noindent{\bf Acknowledgements.} The authors would like to thank Professor
Marcin Bownik and Professor Winfried Sickel for some helpful
discussions on the subject of this paper. The authors also wish to express their
sincerely thanks to the referees for their very carefully reading and
valuable remarks which improve the presentation of this article.

\bigskip

Dachun Yang and Wen Yuan

\medskip

School of Mathematical Sciences, Beijing Normal University,
Laboratory of Mathematics and Complex Systems, Ministry of
Education, Beijing 100875, People's Republic of China

\smallskip

{\it E-mails}: \texttt{dcyang@bnu.edu.cn} (D. Yang)

\hspace{1.55cm}\texttt{wenyuan@bnu.edu.cn} (W. Yuan)

\end{document}